\documentclass{article}

\usepackage[english]{babel}

\usepackage[letterpaper,top=2cm,bottom=2cm,left=3cm,right=3cm,marginparwidth=1.75cm]{geometry}

\usepackage{amssymb,mathrsfs,amscd}
\usepackage{amsthm}
\usepackage{hyperref}
\usepackage{amsmath}
\usepackage{lineno,hyperref,color}
\usepackage{amsmath}
\usepackage{graphicx}

\title{Kato's chaos of multiple mappings and its continuous self-maps}
\author{Yingcui Zhao}

\begin{document}
\newtheorem{theorem}{Theorem}[section]
\newtheorem{corollary}[theorem]{Corollary}
\newtheorem{lemma}[theorem]{Lemma}
\newtheorem{proposition}[theorem]{Proposition}
\newtheorem{problem}[theorem]{Problem}
\newtheorem{maintheorem}[theorem]{Main Theorem}
\newtheorem{definition}[theorem]{Definition}
\newtheorem{remark}[theorem]{Remark}
\newtheorem{example}[theorem]{Example}
\newtheorem{claim}{Claim}[section]
\maketitle

\begin{abstract}
In 2016, Hou and Wang\cite{HouWang2016} introduced the concept of multiple mappings based on iterated function system, which is an important branch of fractal theory. In this paper, we introduce the definitions of sensitivity, accessibility, and Kato's chaos of multiple mappings from a set-valued perspective. We show that multiple mappings and its continuous self-maps do not imply each other in terms of sensitivity and accessibility.
While a sufficient condition for multiple mappings to be sensitive, accessible and Kato's chaotic is provided, respectively.  And the sensitivity, accessibility, and Kato's chaos of multiple mappings are preserved under topological conjugation.
\end{abstract}

\section{Introduction}
Chaos is one of the most important research topics in the field of topological dynamical systems. The first rigorous mathematical definition of chaos can be traced back to 1975 by Li and Yorke\cite{ly}. Since then, scholars have embarked on vigorous research on chaos. Scholars from different fields integrated the study of chaos into their own research directions, describing the properties of chaos from different perspectives. Then, a series of different concepts of chaos have emerged, such as 
Kato chaos\cite{Kato1996}, Devaney chaos\cite{Devaney1989}, Auslander-Yorke chaos\cite{AuslanderYorke1980}, distribution chaos\cite{SS1994}
 and others(see \cite{Wang2022, WangHuang2007}, for example). Among the many existing concepts of chaos, Kato chaos is a very important concept.

Suppose that $X$ is a metric space with a metric $d$,  and $f$ is a continuous self-map on $X$. Let $\mathbb{N}=\{0,1,2,\cdots\}$, $\mathbb{Z^+}=\{1,2,3,\cdots\}$. We say that the dynamical system $(X,f)$ (or the map $f$) is 
\begin{enumerate}	
	\item[(1)] \emph{sensitive}, if there exists $\delta>0$ such that for any nonempty open set $U\subset X$, there exist $x,y\in U$ and $n\in \mathbb{Z^+}$ such that $d(f^n(x),f^n(y))>\delta$. 
	
	\item[(2)] \emph{accessible}, if for any $\epsilon >0$ and any nonempty open sets $U, V\subset X$ there exist $x\in U$, $y
	\in V$ and $n\in \mathbb{Z^+}$ such that $d(f^n(x),f^n(y))<\epsilon $. 
	
	\item[(3)]\emph{Kato chaotic}, if it is sensitive and accessible.
\end{enumerate}
Kato's chaos is also said to be everywhere chaotic by H. Kato in \cite{Kato1996}. H. Kato' also provided an equivalent description of Kato' chaos for continuous maps on compact metric spaces using a method similar to Li-Yorke chaos and proved that topological mixing implies Kato' chaos. In \cite{Gu2007}, Gu showed that for a continuous self-map on complete metric space with a fixed point but without an isolated point, Ruelle-Takens chaos implies Kato' chaos. But the converse isn't always true. Thus, one can see that Ruelle-Takens chaos is strictly stronger than Kato's chaos. In \cite{LiLu2020}, Li et al. studied the relations between topological weak mixing and collective accessibility, or strong accessibility, or strong Kato chaos. 

In 2016, Hou and Wang\cite{HouWang2016} defined multiple mappings derived from iterated function system. Iterated function system is an important branch of fractal theory, reflecting the essence of the world \cite{Attia2023}. They are one of the three frontiers of nonlinear science theory. Hou and Wang's focus focus was primarily on studying the Hausdorff metric entropy and Friedland entropy of multiple mappings. Additionally, they introduced the notions of Hausdorff metric Li-Yorke chaos and Hausdorff metric distributional chaos from a set-valued perspective. It is worth noting that researchers studying iterated function systems often approach the topic from a group perspective rather than a set-valued perspective. This also establishes a close connection between multiple mappings and set-valued mappings, or we can consider multiple mappings as a special case of set-valued mappings.

It is important to acknowledge the valuable role of set-valued mappings in addressing complex problems involving uncertainty, ambiguity, or multiple criteria. Set-valued mappings offer versatility and flexibility, making them highly beneficial across various fields. One prominent application of set-valued mappings is in optimization problems, where the objective is to identify the optimal set of solutions \cite{Huang2023}. For instance, in multi-objective optimization, a set-valued mapping can represent the Pareto front, encompassing all non-dominated solutions.

Set-valued mappings also prove useful in decision-making processes that require considering multiple criteria or preferences \cite{Zor2021}. By representing feasible solutions as sets, decision-makers can thoroughly analyze and compare different options, enabling them to make well-informed decisions. Additionally, set-valued mappings find applications in data analysis tasks such as clustering and classification. Unlike assigning each data point to a single category, set-valued mappings can represent uncertainty or ambiguity by assigning data points to multiple categories. In fact, the applications of set-valued mappings are vast and diverse, encompassing numerous fields beyond those mentioned here.

In \cite{ziji}, we studied that if multiple mappings $F$ has a disdributionally chaotic
pair, especially $F$ is distributionally chaotic, $S\Omega(F)$
contains at least two points and gives a sufficient condition for
$F$ to be distributionally chaotic in a sequence and chaotic in the
strong sense of Li-Yorke. Zeng et al. \cite{zeng2020} proved two topological conjugacy dynamical systems to multiple mappings have simultaneously Hausdorff metric Li-Yorke chaos or Hausdorff distributional chaos and the multiple mappings $F$ and its $2$-tuple of continuous self-maps $f_1,f_2$ are not mutually implied in terms of Hausdorff metric Li-Yorke chaos.

The current paper aims to consider the image of one point under multiple mappings as a set (a compact set). We primarily consider the relationship between multiple mappings $F$ and its $2$-tuple of continuous self-maps $f_1,f_2$ in terms of sensitivity, accessibility and Kato's chaos. 

The specific layout of the present paper is
as follows. Some preliminaries and definitions are introduced in Section \ref{sec2}.
Then we study the relation between multiple mappings and its continuous self-maps in terms of sensitivity, accessibility, and Kato's chaos in Section \ref{sec3}, and show that the sensitivity, accessibility, and Kato's chaos of multiple mappings are preserved under topological conjugation in Section \ref{sec4}. 

\section{Preliminaries}\label{sec2}
Let $F=\{f_1,f_2,\cdots,f_n\}$ be a \emph{multiple mappings} with $n$-tuple of continuous self-maps on $X$. Note that for any $x\in X$, $F(x)=\{f_1(x),f_2(x),\cdots,f_n(x)\}\subset X$ is compact. Let $$\mathbb{K}(X)=\{K\subset X\mid K~\text{is nonempty and compact}\}.$$
Then $F$ is from $X$ to $\mathbb{K}(X)$. The Hausdorff metric $d_H$ on $\mathbb{K}(X)$ is defined by $$d_H(A,B)=\max\{\sup_{a\in A}\inf_{b\in B}d(a,b),\sup_{b\in B}\inf_{a\in A}d(a,b)\},\forall A,B\subset X.$$ It is clear that $\mathbb{K}(X)$ is a compact metric space with the Hausdorff metric $d_H$. 

For convenience, let us study multiple mappings on compact metric spaces from a set-valued view, using the example of examining two continuous self-maps. The definitions and conclusions presented in this paper can be easily extended to the case of a multiple mappings formed by any finite number of continuous self-maps. Consider the multiple mappings $F=\{f_1,f_2\}$. For any $n>0$, $F^n:X\rightarrow\mathbb{K}(X)$ is defined by for  any $x\in X$,$$F^n(x)=\{f_{i_1}f_{i_2}\cdots f_{i_n}(x)\mid i_1,i_2,\cdots,i_n=1\text{ or }2\}.$$
It is obvious that $F^n(x)\in \mathbb{K}(X)$. For any $A\subset X$, let $$F^n(A)=\{f_{i_1}f_{i_2}\cdots f_{i_n}(a)\mid a\in A, i_1,i_2,\cdots,i_n=1\text{ or }2\}=\bigcup_{a\in A}F^n(a).$$
Particularly, if $A\in\mathbb{K}(X)$, $F^n(A)\in\mathbb{K}(X)$. Then, $F$ naturally induces a continuous self-map on $\mathbb{K}(X)$, denoted by $\widetilde{F}:\mathbb{K}(X)\rightarrow\mathbb{K}(X)$.

Now, we define the concept of sensitivity, accessibility and Kato's chaos for multiple mappings from a set-valued view. 
\begin{definition}\label{newde}
	Suppose that $F=\{f_1,f_2\}$ be a multiple mappings with $2$-tuple of continuous self-maps on $X$. We say that $F$ is
	\begin{enumerate}	
		\item[(1)] \emph{(Hausdorff metric) sensitive}, if there exists $\delta>0$ such that for any nonempty open set $U\subset X$, there exist $x,y\in U$ and $n\in \mathbb{Z^+}$ such that $d_H(F^n(x),F^n(y))>\delta$. 
		
		\item[(2)] \emph{(Hausdorff metric) accessible}, if for any $\epsilon >0$ and any nonempty open sets $U, V\subset X$ there exist $x\in U$, $y
		\in V$ and $n\in \mathbb{Z^+}$ such that $d_H(F^n(x),F^n(y))<\epsilon $. 
		
		\item[(3)]\emph{(Hausdorff metric) Kato chaotic}, if it is sensitive and accessible.
	\end{enumerate}
\end{definition}
It is easy to see that the Hausdorff metric sensitivity, accessibility and Kato's chaos of multiple mappings, in the case of degradation (where the multiple mappings consists of only one continuous self-map), is the same as the sensitivity of a classical single continuous self-map. Next we provide an example to illustrate the existence of the newly defined concept Definition \ref{newde}. 
\begin{example}
	Consider the multiple map defined on $[0,1]$ as $F=\{f_1,f_2\}$, in which 
	\begin{equation*}
		\begin{aligned}
			f_1(x)=\begin{cases}
				2x, & 0\leq x\leq\frac{1}{2}, \\
				2-2x, & \frac{1}{2}<x\leq 1,
			\end{cases}
		\end{aligned}
		\begin{aligned}
			f_2(x)=\begin{cases}
				1-2x, & 0\leq x\leq\frac{1}{2}, \\
				2x-1, & \frac{1}{2}<x\leq 1.
			\end{cases}
		\end{aligned}
	\end{equation*}
	\begin{itemize}
		\item[(1)]It is can be verified that for any $x\in [0,1]$ and any $n>0$, $F^n(x)=\{f_1^n(x),f_2^n(x)\}$ and $f_1^n(x)+f_2^n(x)=1$.  Let $U$ be nonempty open set of $[0,1]$. Then there exist $x,y\in U$ and $n\neq m$ such that $f_1^n(x)=f_2^n(x)=\frac{1}{2}$ and $f_1^m(y)=f_2^m(y)=\frac{1}{2}$. Without loss of generality, let $n>m$. Then, $F^n(x)=\{\frac{1}{2}\}$ and $F^n(y)=\{0,1\}$. Thus, $d_H(F^n(x),f^n(y))>\frac{1}{2}$. So, $F$ is sensitive.
		
		\item[(2)]Let $\epsilon>0$ and $U,V$ be nonempty open sets of $[0,1]$. Then there exist $x\in U,y\in V$ and $n,m>0$ such that $f_1^n(x)=f_2^n(x)=\frac{1}{2}$ and $f_2^m(y)=f_1^m(y)=\frac{1}{2}$. Thus there exists $k=\max\{n,m\}+2$ such that $F^k(x)=F^k(y)=\{0,1\}$. That is $d_H(F^n(x),f^n(y))=0<\epsilon$. So, $F$ is accessible.
		\item[(3)] By (1) and (2), $F$ is Kato chaotic.
	\end{itemize}
\end{example}
\section{Relation Between $F$ and $f_1,f_2$}\label{sec3}
A natural question is what is the implication between the sensitivity / accessibility / Kato's chaos of multiple mappings $F=\{f_1,f_2\}$ and the sensitivity / accessibility / Kato's chaos of its $2$-tuple of continuous self-maps $f_1,f_2$?

\subsection{Sensitivity}
Firstly, we show the sensitivity of $F$ can't imply that $f_1$ or $f_2$ is sensitive.
\begin{example}\label{yyzeng}
	Consider the multiple mappings defined on $[0,1]$ as $F=\{f_1,f_2\}$, in which
	
	\begin{equation*}
		\begin{aligned}
			f_1(x)=\begin{cases}
				2x, & 0\leq x\leq\frac{1}{2}, \\
				1, & \frac{1}{2}<x\leq 1,
			\end{cases}
		\end{aligned}
		\begin{aligned}
			f_2(x)=\begin{cases}
				1, & 0\leq x\leq\frac{1}{2}, \\
				2-2x, & \frac{1}{2}<x\leq 1.
			\end{cases}
		\end{aligned}
	\end{equation*}
	Let
	\begin{equation*}
		f(x)=\begin{cases}
			2x, & 0\leq x\leq\frac{1}{2}, \\
			2-2x, & \frac{1}{2}<x\leq 1.
		\end{cases}
	\end{equation*}
	Then for any $x\in [0,1]$ and any $n\geq 2$, $F^n(x)=\{0,1,f^n(x)\}$. 
	\begin{itemize}
		\item[(1)]As we all know, $f$ is sensitive. Let $\delta>0$ be the sensitive constant for $f$. Then for any nonempty open set $U\subset [0,1]$, there exist $x,y\in U$ and $n>0$ such that $d(f^n(x),f^n(y))>\delta$. Therefore, there exist $x,y\in U$ and $n>0$ such that $d_H(F^n(x),F^n(y))>\delta$. So, $F$ is sensitive.
		\item[(2)]Let $U=(\frac{1}{2},1)$. Then for any $x,y\in U$ and any $n>0$, $f_1^n(x)=f_1^n(y)=1$. So, $f_1$ is not sensitive.
		\item[(3)]Let $U=(0,\frac{1}{2})$. Then for any $x,y\in U$ and any $n>0$, $f_2^n(x)=f_2^n(y)=1$ or $0$. So, $f_2$ is not sensitive.
	\end{itemize}
\end{example}
Although $F$ being sensitive does not imply $f_1$ or $f_2$ being sensitive, the following theorem demonstrates that this implication holds under additional conditions.
\begin{theorem}
	If $F$ is sensitive and $f_1(x)=c$ ($\forall x\in X$, $c$ is a constant) , then $f_2$ is sensitive.
\end{theorem}
\begin{proof}
	Let $\delta>0$ be a sensitive constant of $F$, $x\in X$ and $\epsilon>0$. Then there exist $y\in B_d(x,\epsilon)$ and $n>0$ such that $d_H(F^n(x),F^n(y))>\delta$. That is, $$d_H(\{c,f_2(c),f_2^2(c),\cdots,f_2^{n-1}(c),f_2^n(x)\},\{c,f_2(c),f_2^2(c),\cdots,f_2^{n-1}(c),f_2^n(y)\})>\delta.$$ Then $d(f_2^n(x),f_2^n(y))>\delta$. So, $f_2$ is sensitive.
\end{proof}

Considering this implication in reverse, we give a sufficient condition for $F$ to be sensitive as Theorem \ref{thFsen}.

\begin{theorem}\label{thFsen}
	If $f_1(x)=c$ ($\forall x\in X$, $c$ is a constant) , $f_2(c)=c$, and $f_2$ is sensitive, then $F$ is sensitive.
\end{theorem}
\begin{proof}
	For any $x\in X$ and any $n>0$, $F^n(x)=\{c,f_2^n(x)\}$. Let $\delta>0$ be a sensitive constant of $f_2$. Then for any $x\in X$ and any $\epsilon>0$, there exist $y_0\in B_d(x,\epsilon)$ and $n_0>0$ such that $d(f_2^{n_0}(x),f_2^{n_0}(y_0))>\delta$. 
	
	Suppose that $\frac{\delta}{2}$ is not a sensitive constant of $F$. Then there exist $x_0\in X$ and $\epsilon_0>0$ such that for any $y\in B_d(x_0,\epsilon_0)$ and any $n>0$, we have $$d_H(F^n(x_0),F^n(y))=d_H(\{c,f_2^n(x_0)\},\{c,f_2^n(y)\})\leq\frac{\delta}{2}.$$ Thus at least one of the following two statements holds true.
	
	\begin{itemize}
		\item [(1)]$d(f_2^n(x_0),f_2^n(y))\leq\frac{\delta}{2}$.
		\item[(2)]$d(f_2^n(x_0),f_2^n(y))>\frac{\delta}{2}$, $d(f_2^n(x_0),c)\leq\frac{\delta}{2}$ and $d(c,f_2^n(y))\leq\frac{\delta}{2}$. 
	\end{itemize}
	Then for any $y\in B_d(x_0,\epsilon_0)$ and any $n>0$, $d(f_2^n(x_0),f_2^n(y))\leq\delta$. This is a  contradiction. So, $F$ is sensitive.
\end{proof}

\subsection{Accessibility}
Firstly, we show that there exists a multiple mappings $F=\{f_1,f_2\}$ that is accessible, but neither $f_1$ nor $f_2$ is accessible.
\begin{example} Consider the multiple mappings defined on $X=\{0,1,2\}$ as $F=\{f_1,f_2\}$, in which $f_1:0\rightarrow 1\rightarrow 2\rightarrow 0$ and $f_2:0\rightarrow2\rightarrow1\rightarrow0$. 
	
	\begin{itemize}
		\item[(1)]For any $x\in X$ and any $n\geq2$, $F^n(x)=X$. So, $F$ is accessible.
		\item[(2)]For any $n>0$, $f_1^n(0)\neq f_1^n(1)$ and $f_2^n(0)\neq f_2^n(1)$. So, neither $f_1$ nor $f_2$ is accessible.
	\end{itemize}
\end{example}
Next we show there exists an example in which both $f_1$ and $f_2$ are accessible but $F=\{f_1,f_2\}$ isn't accessible. While before that, let us introduce two necessary lemmas.

\begin{lemma}\label{sas2a}
	Let $X$ be a nonempty invariable set of $\Sigma^2$. If $\sigma$ is accessible, then $\sigma^2$ is accessible.
\end{lemma}
\begin{proof}
	Necessity$\Rightarrow$: Let $\epsilon>\frac{1}{K}>0$, $U$ and $V$ be nonempty open sets of $X$. By $\sigma$ is accessible, there exist $x\in U$, $y\in V$ and $n>0$ such that $d(\sigma^n(x),\sigma^n(y))<\frac{1}{K+4}$. Then $d((\sigma^2)^{[\frac{n}{2}]+1}(x),(\sigma^2)^{[\frac{n}{2}]+1}(y))<\frac{1}{K}<\epsilon$. So, $\sigma^2$ is accessible.
	
	Sufficiency is easy.
\end{proof}
\begin{lemma}\label{dab}
	Let $A,B\subset \mathbb{N}$. If $d(A)=d(B)=1$, then $d(A\cap B)=1$.
\end{lemma}
\begin{proof}
	Put $a_{2n}=\frac{1}{2n}\mid A\bigcap\{0, 1, \cdots, 2n-1\}\mid $, $a_{2n+1}=\frac{1}{2n+1}\mid A\bigcap\{0, 1, \cdots, 2n\}\mid $, $b_{2n}=\frac{1}{2n}\mid B\bigcap\{0, 1, \cdots, 2n-1\}\mid $, $b_{2n+1}=\frac{1}{2n+1}\mid B\bigcap\{0, 1, \cdots, 2n\}\mid $, $c_{2n}=\frac{1}{2n}\mid (\mathbf{N}-B)\bigcap\{0, 1, \cdots, 2n-1\}\mid $, $c_{2n+1}=\frac{1}{2n+1}\mid (\mathbf{N}-B)\bigcap\{0, 1, \cdots, 2n\}\mid $. Then,  $\lim_{n\rightarrow\infty}a_{2n}=\lim_{n\rightarrow\infty}a_{2n+1}=\lim_{n\rightarrow\infty}b_{2n}=\lim_{n\rightarrow\infty}b_{2n+1}=1,$ and $\lim_{n\rightarrow\infty}c_{2n}=\lim_{n\rightarrow\infty}c_{2n+1}=0.$ Thus, $\lim_{n\rightarrow\infty}\frac{1}{2n}\mid (A\bigcap B)\bigcap\{0, 1, \cdots, 2n-1\}\mid \geq\lim_{n\rightarrow\infty}(a_{2n}-c_{2n})=1,$ and $\lim_{n\rightarrow\infty}\frac{1}{2n+1}\mid (A\bigcap B)\bigcap\{0, 1, \cdots, 2n\}\mid \geq\lim_{n\rightarrow\infty}(a_{2n+1}-c_{2n+1})=1$.
	So, $d(A\bigcap B)=1$.
\end{proof}
\begin{example}
	Let $\theta=0\cdots0\cdots$, $$\omega=\overbrace{11}^{2^1}\overbrace{0}^{1}\overbrace{1\cdots1}^{2^2}\overbrace{00}^{2}\overbrace{1\cdots1}^{2^3}\overbrace{000}^{3}\cdots\cdots\overbrace{1\cdots1}^{2^n}\overbrace{00}^{n}\cdots\cdots,$$ and $$X=\{\theta\}\cup\{\omega,\sigma(\omega),\sigma^2(\omega),\cdots \cdots\}=\{\theta\}\cup orb(\omega,\sigma).$$
	Consider the multiple mappings defined on $X$ as $F=\{\sigma,\sigma^2\}$, in which $\sigma$ is a shift mapping $\sigma(x)=x_1x_2\cdots,\forall x=x_0x_1x_2\cdots\in X$. And the metric on $X$ is defined by for any $x=x_0x_1\cdots,y=y_0y_1\cdots\in X$,
	\begin{equation*}
		d(x,y)=\begin{cases}
			0, & x=y, \\
			\frac{1}{k}, & x\neq y, 
		\end{cases}
	\end{equation*}
	in which $k=\min\{n\geq 0\mid x_n\neq y_n\}+1.$ 
	\begin{itemize}
		\item [(1)]Let $\epsilon>0$. Then there exists $K\in\mathbb{N}$ such that $\epsilon>\frac{1}{K}>0$. Let $x=x_0x_1x_2\cdots,y=y_0y_1y_2\cdots\in X$.
		\begin{itemize}
			\item [(i)] $x=\theta$ and $y\in orb(\omega,\sigma)$. By the structure of $\omega$, there exists $n>0$ such that $d(\sigma^n(\theta),\sigma^n(y))<\frac{1}{K}<\epsilon$.
			\item[(ii)]$x,y\in orb(\omega,\sigma)$. Let $A=\{i\in\mathbb{N}\mid x_i=1\}$ and $B=\{i\in\mathbb{N}\mid y_i=1\}$. Then $d(A)=d(B)=1$. By Lemma \ref{dab}, $d(A\cap B)=1$. Thus, there exists $n>0$ such that $d(\sigma^n(\theta),\sigma^n(y))<\frac{1}{K}<\epsilon$.
		\end{itemize}
		By (i) and (ii), $\sigma$ is accessible.
		\item[(2)] By (1) and Lemma \ref{sas2a}, $\sigma^2$ is accessible.
		\item[(3)] Firstly, we show $\theta$ is a isolated point of $X$ using proof by contradiction.
		
		Suppose that there exists $\{n_i\}$ such that $\lim_{i\rightarrow\infty}\sigma^{n_i}(\omega)=\theta$. Then for any $\epsilon>0$, there exists $N$, for any $n_i>N$, $d(\sigma^{n_i}(\omega),\theta)<\epsilon$. By the structure of $\omega$, for any $N$, there exists $m>N$ such that $d(\sigma^{m}(\omega),\theta)>\frac{1}{2}$. This is a contradiction. So, $\theta$ is a isolated point of $X$.
		
		Secondly, we show $F$ is not accessible. Let $U=\{\theta\}$ and $V=\{x_0x_1x_2\cdots\in X\mid x_0x_1\cdots x_{19}=11011110011111111000\}=\{\omega\}$, then both $U$ and $V$ are nonempty open sets of $X$. For any $x\in U$ and any $y\in V$, if there exists $n>0$ such that $d_H(F^n(x),F^n(y))<\frac{1}{2}$, then $d(\sigma^n(x),\sigma^n(y))<\frac{1}{2}$. For $y$, the subscripts of $0$ which has a $1$ in front of are (in descending order):
		$$2^1,2^1+2^2+1,2^1+2^2+2^3+1+2,\cdots,
		2(2^m-1)+\frac{m(m-1)}{2},\cdots$$
		
		For any $m\geq 1$, $$\mu=2(2^m-1)+\frac{m(m-1)}{2}\leq n\leq 2(2^{m+1}-1)+\frac{m(m+1)}{2}=\nu.$$ By $\mu>m$ and $\nu>m+1$, there exists $n\leq l\leq 2n$ such that the first element of $\sigma^l(y)$ is $1$. That is, there exists $\sigma^l(y)\in F^n(y)$ for any $a\in F^n(\theta)$, $d(a,\sigma^l(y))\geq\frac{1}{2}$. Then, $$d_H(F^(\theta),F^n(y))\geq\frac{1}{2}.$$
		So, $F$ is not accessible.
	\end{itemize}
\end{example}
Although that both $f_1$ and $f_2$ are accessible can't imply $F=\{f_1,f_2\}$ is accessible, we give two sufficient conditions for $F$ to be accessible as Theorem \ref{2} and Theorem \ref{thFacc}.

\begin{theorem}\label{2}
	If $f_1(x)=c$ ($\forall x\in X$, $c$ is a constant) , $f_2(c)=c$, and $f_2$ is accessible, then $F$ is accessible.
\end{theorem}
\begin{proof}
	For any $\epsilon>0$ and any nonempty open sets $U,V\subset X$, there exist $x\in U$, $y\in V$ and $n\in\mathbb{Z}^+$ such that $d(f_2^n(x),f_2^n(y))<\epsilon$. Since for any $x\in X$ and any $m>0$, $F^m(x)=\{c,f_2^m(x)\}$, 
	$$d_H(F^n(x),F^n(y))=d_H(\{c,f_2^n(x)\},\{c,f_2^n(y)\})<\epsilon.$$ So, $F$ is accessible.
\end{proof}
\begin{theorem}\label{thFacc}
	If there exists $0<\lambda<1$ such that for any nonempty open sets $U,V\subset X$, there exist $x\in U$ and $y\in V$ satisfying $$d(f_i(x),f_i(y))<\lambda d(x,y), \forall i=1,2,$$
	then $F$ is accessible.
\end{theorem}
\begin{proof}
	Let $\epsilon>0$ and $U,V$ be nonempty open sets of $X$. Select $x_0\in U$ and $y_0\in V$ such that $$d(f_{\alpha_1}(x_0),f_{\alpha_1}(y_0))<\lambda d(x_0,y_0),\forall \alpha_1=1,2.$$
	
	\begin{itemize}
		\item [Step $1$:] Take $\eta_1=\frac{\lambda^2 d(x_0,y_0)}{5}>0$. By $f_1$ and $f_2$ are continuous and $X$ is compact, there exists $0<\delta_1<\frac{\lambda d(x_0,y_0)}{4}$ such that for any $x,y\in X$, $$d(x,y)<\delta_1\Rightarrow d(f_i(x),f_i(y))<\eta_1,\forall i=1,2.$$
		
		Select $u_{\alpha_1}\in B(f_{\alpha_1}(x),\delta_1)$ and $v_{\alpha_1}\in B(f_{\alpha_1}(y),\delta_1)$ satisfying $$d(f_{\alpha_2}(u_{\alpha_1}),f_{\alpha_2}(v_{\alpha_1}))<\lambda d(u_{\alpha_1},v_{\alpha_1}),\forall \alpha_2=1,2.$$
		Then 
		\begin{align}
			&d(f_{\alpha_2}f_{\alpha_1}(x_0),f_{\alpha_2}f_{\alpha_1}(y_0)) \nonumber \\   
			<&d(f_{\alpha_2}f_{\alpha_1}(x_0),f_{\alpha_2}(u_{\alpha_1}))+d(f_{\alpha_2}(u_{\alpha_1}),f_{\alpha_2}(v_{\alpha_1}))+d(f_{\alpha_2}(v_{\alpha_1}),f_{\alpha_2}f_{\alpha_1}(y_0))\nonumber\\
			<&\eta_1+\lambda d(u_{\alpha_1},v_{\alpha_1})+\eta_1 \nonumber\\
			<&2\eta_1+\lambda(d(u_{\alpha_1},f_{\alpha_1}(x_0))+d(f_{\alpha_1}(x_0),f_{\alpha_1}(y_0))+d(f_{\alpha_1}(y_0),v_{\alpha_1})) \nonumber\\
			<&2\eta_1+\lambda(\delta_1+\lambda d(x_0,y_0)+\delta_1) \nonumber\\
			<&2\eta_1+\lambda^2 d(x_0,y_0)+2\delta_1\lambda\nonumber\\
			<&\lambda^2 d(x_0,y_0)+2\lambda\frac{\lambda d(x_0,y_0)}{4}+2\frac{\lambda^2 d(x_0,y_0)}{4}\nonumber\\
			=&2\lambda^2 d(x_0,y_0).\nonumber
		\end{align}	
		\item [Step $2$:] Take $\eta_2=\frac{\lambda^3 d(x_0,y_0)}{5}>0$. By $f_1$ and $f_2$ are continuous and $X$ is compact, there exists $0<\delta_2<\frac{\lambda^2 d(x_0,y_0)}{4}$ such that for any $x,y\in X$, $$d(x,y)<\delta_2\Rightarrow d(f_i(x),f_i(y))<\eta_2,\forall i=1,2.$$
		
		Select $u_{\alpha_2}\in B(f_{\alpha_2}f_{\alpha_1}(x),\delta_2)$ and $v_{\alpha_2}\in B(f_{\alpha_2}f_{\alpha_1}(y),\delta_2)$ satisfying $$d(f_{\alpha_3}(u_{\alpha_2}),f_{\alpha_3}(v_{\alpha_2}))<\lambda d(u_{\alpha_2},v_{\alpha_2}),\forall \alpha_3=1,2.$$
		Then 
		\begin{align}
			&d(f_{\alpha_3}f_{\alpha_2}f_{\alpha_1}(x_0),f_{\alpha_3}f_{\alpha_2}f_{\alpha_1}(y_0)) \nonumber \\   
			<&d(f_{\alpha_3}f_{\alpha_2}f_{\alpha_1}(x_0),f_{\alpha_3}(u_{\alpha_2}))+d(f_{\alpha_3}(u_{\alpha_2}),f_{\alpha_3}(v_{\alpha_2}))+d(f_{\alpha_3}(v_{\alpha_2}),f_{\alpha_3}f_{\alpha_2}f_{\alpha_1}(y_0))\nonumber\\
			<&2\eta_2+\lambda d(u_{\alpha_2},v_{\alpha_2}) \nonumber\\
			<&2\eta_2+\lambda(d(u_{\alpha_2},f_{\alpha_2}f_{\alpha_1}(x_0))+d(f_{\alpha_2}f_{\alpha_1}(x_0),f_{\alpha_2}f_{\alpha_1}(y_0))+d(f_{\alpha_2}f_{\alpha_1}(y_0),v_{\alpha_2})) \nonumber\\%
			<&2\eta_2+\lambda(2\delta_2+2\lambda^2 d(x_0,y_0))\nonumber\\%
			<&\frac{2}{4}\lambda^3 d(x_0,y_0)+2\lambda\frac{\lambda^2 d(x_0,y_0)}{4}+2\lambda^3 d(x_0,y_0)\nonumber\\
			=&3\lambda^3 d(x_0,y_0).\nonumber
		\end{align}	
		\item[$\cdots$ $\cdots$]
		
		\item[Step $n$:] The same goes for this step $n$, in which $n$ satisfies $n\lambda^n d(x_0,y_0)<\epsilon$. Then $$d(f_{\alpha_n}\cdots f_{\alpha_1}(x_0),f_{\alpha_n}\cdots f_{\alpha_1}(y_0))<n\lambda^n d(x_0,y_0)<\epsilon,\forall \alpha_1, \alpha_2,\cdots,\alpha_n=1,2.$$ 
		That is, for any $u\in F^n(x_0)$ there exists $v\in F^n(y_0)$ such that $d(u,v)<\epsilon$ and for any $v\in F^n(y_0)$ there exists $u\in F^n(x_0)$ such that $d(u,v)<\epsilon$. Then $$d_H(F^n(x_0),F^n(y_0))<\epsilon.$$ So, $F$ is accessible.
	\end{itemize}
\end{proof}

Now let's illustrate Theorem \ref{thFacc}.
\begin{example}
	Let $S^1$ be the unit circle on the complex plane. Define the metric on $S^1$ by $d(e^{i\alpha},e^{i\beta})=\frac{\mid \alpha-\beta\mid }{2\pi}$ for any $e^{i\alpha},e^{i\beta}\in S^1$. Consider the multiple mappings defined on $S^1$ as $F=\{f_1,f_2\}$, in which $f_1(e^{i\alpha})=e^{i\frac{\alpha}{2}}$ and $f_2(e^{i\alpha})=e^{i\frac{\alpha}{3}}$. 
	
	For any $e^{i\alpha},e^{i\beta}\in S^1$, $$d(f_1(e^{i\alpha}),f_1(e^{i\beta}))=d(e^{i\frac{\alpha}{2}},e^{i\frac{\beta}{2}})=\frac{{\mid \frac{\alpha}{2}-\frac{\beta}{2}\mid }}{2\pi}<\frac{\frac{2}{3}\mid \alpha-\beta\mid }{2\pi}=\frac{2}{3}d(e^{i\alpha},e^{i\beta}).$$ So, $F$ is accessible.
\end{example}
\subsection{Kato's Chaos}
Kato's chaos of $F$ can't imply that $f_1$ or $f_2$ is Kato chaotic.
\begin{example}
	Consider the multiple mappings $F$ which is defined by Example \ref{yyzeng}. 
	
	Let $\epsilon>0$ and $U,V\subset X$ be nonempty and open. Then by the accessibility of the tent map, there exist $x\in U$, $y\in V$ and $n>0$ such that $d_H(F^n(x),F^n(y))<\epsilon$. So, $F$ is accessible. And $F$ is Kato chaotic. While neither $f_1$ nor $f_2$ is sensitive. Then neither $f_1$ nor $f_2$ is Kato chaotic.
\end{example}
Based on Theorem \ref{thFsen} and Theorem \ref{2}, the following corollary can be derived. 
\begin{corollary}
	If $f_1(x)=c$ ($\forall x\in X$, $c$ is a constant), $f_2(c)=c$ and $f_2$ is Kato chaotic, then $F$ is Kato chaotic.
\end{corollary}

\section{Topological Conjugation}\label{sec4}
In this section, we will demonstrate that the sensitivity, accessibility, and Kato's chaos of multiple mappings are preserved under topological conjugation. Before that, let's first introduce the definition of topological conjugacy about multiple mappings.

\begin{definition}
	Let $(X,F)$ and $(Y,G)$ be two dynamical systems to multiple mappings on compact metric spaces. If there exists a homeomorphism map $T:X\rightarrow Y$ such that for any $x\in X$, $TF(x)=GT(x)$, then $T$ is said to be a topological conjugacy from $(X,F)$ to $(Y,G)$.
\end{definition}
\begin{theorem}
	Suppose that $(X,F)$ and $(Y,G)$ be dynamical systems to multiple mappings on compact space with metric $d_H$ and $\rho_H$. Let $T:X\rightarrow Y$ be a topological conjugacy from $(X,F)$ to $(Y,G)$. Then $(X,F)$ is Hausdorff metric sensitive if and only if $(Y,G)$ is Hausdorff metric sensitive.
\end{theorem}
\begin{proof}
	necessity$\Rightarrow$: Let $\delta>0$ be the sensitive constant of $F$. Suppose that $G$ is not sensitive. By Lemma 1 of \cite{zeng2020}, $T^{-1}:Y\rightarrow X$ be a topological conjugacy from $(Y,G)$ to $(X,F)$. For $\delta>0$, since $T^{-1}$ is continuous and both $X$ and $Y$ are compact, there exists $2\delta'>0$ such that for any $A,B\in K(Y)$, 
	\begin{equation}\label{eq1}
		\rho_H(A,B)<2\delta'\Rightarrow d_H(T^{-1}(A),T^{-1}(B))<\delta.
	\end{equation}
	
	Since $G$ is not sensitive, there exists nonempty open set $U\subset Y$ such that for any $x,y\in U$ and any $n>0$, $$\rho_H(G^n(x),G^n(y))\leq \delta'<2\delta'.$$
	By (\ref{eq1}), $d_H(T^{-1}G^n(x),T^{-1}G^n(y))<\delta$. By $T^{-1}G^n=F^nT^{-1}$, $$d_H(F^nT^{-1}(x),F^nT^{-1}(y))<\delta.$$ Since $U$ is open and $T^{-1}$ is homeomorphous, $T^{-1}(U)$ is open. Then for any $u,v\in T^{-1}(U)$ and any $n>0$, $$d_H(F^n(u),F^n(v))<\delta,$$
	which contradicts that $\delta$ is a sensitive constant of $F$. So, $G$ is sensitive.
	
	The proof of sufficiency is similar with the above proof of necessity.
\end{proof}
\begin{theorem}
	Suppose that $(X,F)$ and $(Y,G)$ be dynamical systems to multiple mappings on compact space with metric $d_H$ and $\rho_H$. Let $T:X\rightarrow Y$ be a topological conjugacy from $(X,F)$ to $(Y,G)$. Then $(X,F)$ is Hausdorff metric accessible if and only if $(Y,G)$ is Hausdorff metric accessible.
\end{theorem}
\begin{proof}
	necessity$\Rightarrow$: Let $\epsilon>0$. 
	By $T$ is continuous and both $X$ and $Y$ are compact, there exists $\delta>0$ such that for any $A,B\in K(X)$, 
	\begin{equation}\label{eq2}
		d_H(A,B)<\delta\Rightarrow \rho_H(T(A),T(B))<\epsilon.
	\end{equation}
	Let $U,V\subset X$ be nonempty open sets. Then $T^{-1}(U)$ and $T^{-1}(V)$ are nonempty open sets of $X$. By $F$ is accessible, there exist $x\in T^{-1}(U)$, $y\in T^{-1}(V)$ and $n>0$ such that $$d_H(F^n(x),F^n(y))<\delta.$$
	By (\ref{eq2}), $\rho_H(TF^n(x),TF^n(y))<\epsilon$. By $G^nT=TF^n$, $$d_H(G^nT(x),G^nT(y))<\epsilon.$$ Thus, there exist $u=T(x)\in U$, $v=T(y)\in V$ and $n>0$ such that $$\rho_H(G^n(u),G^n(v))<\epsilon.$$ So, $G$ is accessible.
	
	The proof of sufficiency is similar with the above proof of necessity.
\end{proof}
\begin{corollary}
	Suppose that $(X,F)$ and $(Y,G)$ be dynamical systems to multiple mappings on compact space with metric $d_H$ and $\rho_H$. Let $T:X\rightarrow Y$ be a topological conjugacy from $(X,F)$ to $(Y,G)$. Then $(X,F)$ is Hausdorff metric Kato chaotic if and only if $(Y,G)$ is Hausdorff metric Kato chaotic.
\end{corollary}

\section{Conclusions}
We define and study the sensitivity, accessibility and Kato's chaos of multiple mappings from the perspective of a set-valued view. This perspective is different from the group-theoretic view that has been previously studied in the context of dynamical systems of iterated function systems. We show that
\begin{itemize}
	\item [(1)]sensitivity of multiple mappings $F=\{f_1,f_2\}$ and its $2$-tuple of continuous self-maps $f_1,f_2$ do not imply each other. Also, accessibility of $F=\{f_1,f_2\}$ and $f_1,f_2$ do not imply each other.
	
	\item[(2)]a sufficient condition for $F$ to be sensitive is provided, as well as two sufficient conditions for $F$ to be accessible.
	
	\item[(3)]Kato's chaos of $F$ can't imply that $f_1$ or $f_2$ is Kato chaotic.
	
	\item[(4)]the sensitivity, accessibility, and Kato's chaos of multiple mappings are preserved under topological conjugation.
\end{itemize}
The above conclusions not only deepen our understanding of continuous self-maps but also enable us to use relatively simple continuous self-maps to comprehend relatively complex multiple mappings. We hope that these conclusions can enrich the achievements in the field of multiple mappings and provide some assistance for further in-depth research in this area.

\end{document}